\documentclass{amsart}

\usepackage[all,cmtip]{xy}
\usepackage{todonotes}
\usepackage{comment}
\usepackage{enumitem}
\usepackage{graphicx}
\usepackage{tikz-cd}
\tikzset{commutative diagrams/row sep/normal = 7.5 ex}
\tikzset{commutative diagrams/column sep/normal = 8.5 ex}
\usepackage{multicol}
\usepackage{color}
\usepackage[colorlinks=true,linkcolor=blue, urlcolor=blue, citecolor=blue]{hyperref}
\usepackage[utf8]{inputenc} 

\usepackage{amsmath,amssymb,amsthm,bbm}

\usepackage{array, multirow} 

\newtheorem{theorem}{Theorem}[section]
\newtheorem{lemma}[theorem]{Lemma}

\newtheorem{proposition}[theorem]{Proposition}
\newtheorem{question}[theorem]{Question}

\newtheorem{claimcounted}{Claim}

\theoremstyle{definition}
\newtheorem{definition}[theorem]{Definition}

\newtheorem{remark}[theorem]{Remark}

\newcommand {\NN} {{\mathbb N}} 

\newcommand {\caE} {{\mathcal E}}
\newcommand {\caF} {{\mathcal F}}
\newcommand {\caG} {{\mathcal G}}

\newcommand {\caO} {{\mathcal O}}

\newcommand{\PR}{\mathbb{P}}

\newcommand{\cat}[1]{{\normalfont\textbf{#1}}}

\DeclareMathOperator {\Mor} {{{Mor}}}
\DeclareMathOperator{\MMor}{\mathcal{M}\!{\it or}}
\DeclareMathOperator{\QQuot}{\mathcal{Q}\!{\it uot}}

\DeclareMathOperator{\Hilb}{Hilb}
\DeclareMathOperator{\HHilb}{\mathcal{H}\!{\it ilb}}
\DeclareMathOperator {\Spec} {{{Spec}}}

\DeclareMathOperator {\Hom} {{{Hom}}}

\DeclareMathOperator {\id} {{{id}}}
\DeclareMathOperator {\Bl} {{{Bl}}}

\DeclareMathOperator {\pr} {{pr}}

\newcommand*{\defeq}{\mathrel{\vcenter{\baselineskip0.5ex \lineskiplimit0pt
			\hbox{\scriptsize.}\hbox{\scriptsize.}}}%
	=}

\newcommand{\mathand}{\;\;\hbox{and}\;\;}
\newcommand{\mathfor}{\;\;\hbox{for}\;\;}

\newenvironment{subproof}{\begin{proof}[Proof of claim.]}{%
\end{proof}}

\usepackage[style=alphabetic, minalphanames=3]{biblatex}

\AtEveryBibitem{%
	\clearfield{issn} 
	\clearfield{doi} 
	\clearfield{isbn} 
	
	\ifentrytype{online}{}{
		\clearfield{url}
	}
}

\begin{document}

\title{Properties of schemes of morphisms and applications to blow-ups}

\author{Lucas das Dores}
\address{
	IMPA, Estrada Dona Castorina 110, 22460-320, Rio de Janeiro, Brazil
	\textsc{\newline \indent 
		\href{https://orcid.org/0000-0002-8713-5308%
		}{\includegraphics[width=1em,height=1em]{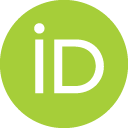} {\normalfont https://orcid.org/0000-0002-8713-5308}}
		}}
\email{lucas.dores@impa.br}

\date{}

\begin{abstract}
	Let $X$ be a fixed projective scheme which is flat over a base scheme $S$. The association taking a quasi-projective $S$-scheme $Y$ to the scheme parametrizing $S$-morphisms from $X$ to $Y$ is functorial. We prove that this functor preserves limits, and both open and closed immersions. As an application, we determine a partition of schemes parametrizing rational curves on the blow-ups of projective spaces at finitely many points. We compute the dimensions of its components containing rational curves outside the exceptional divisor and the ones strictly contained in it. Furthermore, we provide an upper bound for the dimension of the irreducible components intersecting the exceptional divisors properly.
\end{abstract}

\maketitle

\section{Introduction}
\label{aaa}
Let $Y$ be a quasi-projective variety over an algebraically closed field $k$, and let $\Mor(\PR^1, Y)$ be the scheme parametrizing morphisms from $\PR^1$ to $Y$. Suppose that $f:Y' \rightarrow Y$ is a morphism over $k$. Then, there is a natural induced morphism $f_M: \Mor(\PR^1,Y') \rightarrow \Mor(\PR^1, Y)$ given by composition. We investigate a natural question, asked by V. Guletski\u{\i}:

\begin{question}
	\label{aab}
	Which properties of the induced morphism $f_M$ can be deduced from the properties of the morphism $f$?
\end{question}

An answer to this question yields valuable tools to compare spaces parametrizing rational curves on varieties. 

We provide a partial answer in a more general setting. Before stating it, we point out that we use the same notion of projectivity as in \cite{hartshorne-AG}. Namely, an $S$-scheme $X$ is \emph{projective} if there is a closed immersion $X \hookrightarrow \PR^n_S$ over $S$ for some $n \ge 0$ and it is \emph{quasi-projective} if it factors into an open immersion $X \hookrightarrow X'$ with $X'$ being a projective $S$-scheme. With this terminology we prove the following theorem.

\begin{theorem}
	\label{aad}
	Let $S$ be a Noetherian scheme and $X$ be a projective $S$-scheme which is flat over $S$. Let $\cat{Sch}/S$ and $\cat{QProj}/S$ denote the categories of $S$-schemes and quasi-projective $S$-schemes respectively. Then, there is a well defined functor
	\[
	\Mor_S(X, -): \cat{QProj}/S \rightarrow \cat{Sch}/S
	\] 
	associating each quasi-projective $S$-scheme $Y$ to the scheme $\Mor_S(X,Y)$ parametrizing $S$-morphisms from $X$ to $Y$. This functor preserves limits, open and closed immersions.
\end{theorem}

As an application we describe a natural partition of the scheme parametrizing rational curves on the blow-up of a projective space at finitely many points. This follows from several applications of the theorem above using the functor $\Mor(\PR^1, -)$. To state it we introduce the following simple terminology.

\begin{definition}
	\label{dad}
	Let $f: \PR^1 \rightarrow \PR^n$ be a non-constant morphism. Then $f$ is defined by a $(n+1)$-tuple of homogeneous polynomials in $k[u,v]$, denoted $(F_0:\dotsc: F_n)$, with no factors in common and of the same degree. We define the \emph{degree} of the morphism $f$ to be
	\[
	\deg(f) \defeq \deg F_i.
	\]
	Further, let $p$ be a point in $\PR^n$. Up to change of coordinates on $\PR^n$, we may assume $p=(1:0: \dotsb : 0)$. It follows that $p \in f(\PR^1)$ if and only if $\{F_i\}_{1 \le i \le n}$ have factors in common. This is equivalent to saying that there exists a homogeneous polynomial $H \in k[u,v]$ such that $H$ divides $F_i$ for all $1 \le i \le n$. We define
	\[
	m_{p}(f) = \max\{ \deg H \mid H \in k[u,v] \text{ and } H \text{ divides } F_j \text{ for  all } 1 \le j \le n \}
	\]
	to be the \emph{parametric multiplicity} of $p$ on $f$.
\end{definition}

Notice that the parametric multiplicity of $p$ on $f$ is a multiple of the multiplicity of $p$ at the scheme theoretic image of $f$. More precisely, let $f: \PR^1 \rightarrow \PR^n$ be a non-constant morphism of degree $d$, and let $C$ be its scheme theoretic image. If $f$ is birational to its image, then $\deg(f) = \deg(C) = d$ in $\PR^n$ and the parametric multiplicity $m_p(f)$ coincides with the multiplicity of the point $p$ at $C$, denoted $\mu_p(C)$.

If $f$ is not birational to its image, let $\nu: \PR^1 \rightarrow C$ be the normalization of $C$. Then, there is a morphism $g: \PR^1 \rightarrow \PR^1$ such that $f = \nu \circ g$. The morphism $g$ is the composition of a cover and, if the characteristic of $k$ is positive, a Frobenius endomorphism. Clearly, we have
\begin{equation*}
	\deg(f) = \deg(g)\deg(C) \> \mathand \> m_p(f) = \deg(g) \mu_p(C).
\end{equation*}
	
Once we have the definition of parametric multiplicity, let $\sigma: Y \rightarrow \PR^n$ be the blow-up of $\PR^n$ at finitely many points. The main tool we use to describe $\Mor(\PR^1, Y)$ is the study of the morphism
\[
\sigma_M: \Mor(\PR^1, Y) \rightarrow \Mor(\PR^1, \PR^n).
\]
Recall that
\[
\Mor(\PR^1, \PR^n) = \coprod_{d \ge 0} \Mor_d(\PR^1, \PR^n),
\]
where $\Mor_d(\PR^1,\PR^n)$ is a closed subscheme parametrizing morphisms of degree $d$. It is well known that $\Mor_d(\PR^1, \PR^n)$ is irreducible and nonsingular of dimension $(n+1)(d+1)-1$, see for instance \cite[Section 2.1]{debarre-HDAG}.

For any variety $Y$, let $\Mor_{>0}(\PR^1,Y)$ denote the closed subscheme of $\Mor(\PR^1,Y)$ parametrizing non-constant morphisms from $\PR^1$ to $Y$. We also denote by $[g]$ a $k$-point in $\Mor(\PR^1,Y)$ corresponding to a morphism $g: \PR^1 \rightarrow Y$. Then, our main result is the following.

\begin{theorem}
	\label{aac}
	Let $\{p_1, \dotsc, p_r\} \subset \PR^n$ be a finite collection of points in a projective space. Let $\sigma: Y \rightarrow \PR^n$ be the blow-up of $\PR^n$ along these points with exceptional divisor $E$. Let $\mathbf{m} \defeq (m_1, \dotsc, m_r)$ denote an $r$-tuple of non-negative integers. Then, we have the partition in closed subschemes
	\[
	\Mor_{>0}(\PR^1, Y) \cong \left( \coprod_{d > 0} \coprod_{m_i \le d} M_{d,\mathbf{m}} \right) \amalg \Mor_{>0}(\PR^1, E),
	\]
	where
	\begin{enumerate}
		\item \label{dah} a $k$-point $[g]$ belongs to $M_{d, \mathbf{m}}$ if and only if
		\begin{align*}
		\deg (\sigma \circ g) = d  \mathand m_{p_i}(\sigma \circ g) = m_i \mathfor 1 \le i \le r;
		\end{align*}
		
		\item \label{dai} $\Mor_{>0}(\PR^1, E) \cong \coprod_{i=1}^r \coprod_{e>0} \Mor_e(\PR^1, \PR^{n-1})$.
	\end{enumerate} 
	In particular,
	\begin{enumerate}
		\setcounter{enumi}{2}
		\item \label{daj} for $\mathbf{0} \defeq (0, \dotsc, 0)$ and each positive integer $d$, the subschemes $M_{d, \mathbf{0}}$ are nonsingular of dimension $(n+1)(d+1)-1$ and parametrize morphisms whose images do not intersect $E$.
		\item \label{dak} if $\mathbf{m} \neq \mathbf{0}$, every irreducible component of $M_{d,\mathbf{m}}$ has dimension strictly smaller than $\dim M_{d,\mathbf{0}}$.
	\end{enumerate}
\end{theorem}

\vskip 0.5\baselineskip

\textbf{Structure of the paper.} In Section \ref{baa} we recall the definition of Hilbert schemes and schemes of morphisms via their functor of points and deduce properties of natural transformations between these functors to prove Theorem \ref{aad}. In Section \ref{daa} we study schemes parametrizing morphisms from curves to blow-ups of projective schemes at closed subschemes. We recall simple properties of blow-ups of projective spaces at points to prove Theorem \ref{aac}.

\vskip 0.5\baselineskip

\textbf{Acknowledgements.} I would like to thank Vladimir Guletski\u{\i} for suggesting the problem and for many useful discussions. I am grateful to Thomas Eckl, Roy Skjelnes and the anonymous referee for helpful suggestions.

This work was supported by CNPq, National Council for Scientific and Technological Development under the grant [159845/2019-0].

\section{Properties of schemes of morphisms}
\label{baa}

Let $X$ be a separated scheme over a Noetherian scheme $S$ and let $\cat{Noe}/S$ be the category of locally Noetherian schemes over $S$. Recall that the \emph{Hilbert functor}
\[
\HHilb_S(X): (\cat{Noe}/S)^{op} \rightarrow \cat{Set} 
\]
is defined on $S$-schemes as
\[
S' \mapsto \left\{ \begin{array}{c}
V \hookrightarrow X_{S'} \text{ closed subscheme} \\
\text{such that } V \rightarrow S' \text{ is flat and proper}
\end{array} \right\}
\]
and is defined on morphisms by base change. If $X$ is a projective scheme over $S$, this functor is representable by a scheme called the \emph{Hilbert scheme} of $X$ and denoted $\Hilb_S(X)$, see for instance \cite[Theorem 5.16]{nitsure-05} or \cite[Theorem I.1.4]{kollar-RCAV}.

Notice that if $\iota: W \hookrightarrow X$ is an immersion of schemes over $S$, then $\iota$ induces a monomorphism of functors
\begin{equation}
\label{baq}
\eta_\iota: \HHilb_S(W) \hookrightarrow \HHilb_S(X).
\end{equation}
Indeed, it suffices to check this sectionwise: since $X$ is separated, for any $S$-scheme $S'$ and $V \in \HHilb_S(W)(S')$ it follows that the composition $V \hookrightarrow W_{S'} \overset{\iota_{S'}}{\hookrightarrow} X_{S'}$ is proper, hence it is a closed immersion which is flat over $S'$ by assumption. Furthermore, injectivity follows since $\iota_{S'}$ is an immersion.

\begin{definition}
	\label{bae}
	Let $S$ be a scheme and $\cat{C}$ be a full subcategory of the category $\cat{Sch}/S$ of schemes over $S$. For each scheme $Y$ in $\cat{C}$ let $h_Y$ denote its functor of points.
	
	Let $\alpha: \caF \hookrightarrow \caG$ be a monomorphism of contravariant functors on $\cat{C}$. We say $\caF$ is \emph{open} (resp. \emph{closed}) if for every scheme $Y$ in $\cat{C}$ and morphism $\beta: h_Y \rightarrow \caG$, there exists an open (resp. closed) subscheme $\iota_{\beta}: U_{\beta} \hookrightarrow Y$ in $\cat{C}$ such that the fiber product $\caF \times_\caG h_Y$ is naturally isomorphic to $h_{U_\beta}$.
\end{definition}

\begin{remark}
	\label{baf}
	The definition above can be stated in an equivalent way which is more suitable for applications as follows. A monomorphism $\alpha: \caF \hookrightarrow \caG$ is an open (resp. closed) subfunctor if for each $S$-scheme $Y$ and section $V \in \caG(Y)$, there exists an open (resp. closed) subscheme $U_V \hookrightarrow Y$ in $\cat{C}$ satisfying the following universal property:
	\begin{equation*}
	\label{zacd}
	\begin{split}
	\text{a morphism } f: X \rightarrow Y \text{ in } \cat{C} & \text{ factors through } U_V \\
	\text{ if and only if } \caG(f)(V) & \in \alpha_X(\caF(X)).
	\end{split}
	\end{equation*}
\end{remark}

\begin{proposition}
	\label{bac}
	Let $S$ be a Noetherian scheme, $X$ be a separated $S$-scheme and let $j: U \rightarrow X$ be an open immersion of schemes over $S$. Then $\HHilb_{S}(U)$ is an open subfunctor of $\HHilb_{S}(X)$. In particular, if $X$ is projective over $S$, then $\HHilb_{S}(U)$ is representable by an open subscheme $\Hilb_{S}(U) \hookrightarrow \Hilb_{S}(X)$.
\end{proposition}
\begin{proof}
	We have a monomorphism $\eta_j:\HHilb_S(U) \hookrightarrow \HHilb_S(X)$. To prove it is an open subfunctor, it suffices to prove that for every $S$-scheme $S'$ and section in $\HHilb_S(X)(S')$ given by a closed subscheme $i: V \hookrightarrow X_{S'}$, there exists an open subscheme $S_0 \hookrightarrow S'$ satisfying the property:
	\begin{equation}
	\label{bad}
	\begin{split}
	\text{for any } S\text{-morphism } & f: S'' \rightarrow S', f \text{ factors through } S_0 \text{ if and} \\ \text{only if the base change } & i_{S''}: V \times_{S'} {S''} \hookrightarrow X_{S''} \text{ factors through } U_{S''}.
	\end{split}
	\end{equation}
	
	Consider the scheme theoretic intersection $U_{S'} \cap V$ fitting in the commutative diagram
	\[
	\begin{tikzcd}
	U_{S'} \cap V \ar[r, hook, "i'"] \ar[d, hook, "j_{S'}'"' ] & U_{S'} \ar[d, hook, "j_{S'}"] &\\
	V \ar[r, hook, "i"] \ar[rr, bend right =12, "g"'] & X_{S'} \ar[r] & S',
	\end{tikzcd}
	\]
	where the square is Cartesian. Since $g$ is proper, the subset \[
	S_0 = S' \smallsetminus g(V \smallsetminus j_{S'}'(U_{S'} \cap V))
	\]
	is open and it is straightforward to check that 
	\begin{equation}
	\label{bag}
	S_0 = \{p \in S' \mid i_p: V_p \hookrightarrow X_{S',p} \text{ factors through } U_{S',p} \}
	\end{equation}
	where $V_p, X_{S',p}, U_{S',p}$ denote the fibers of $V, X_{S'}$ and $U_{S'}$ over a point $p \in S'$ respectively.
	
	We claim $S_0$ satisfies property \eqref{bad}. Indeed, let $f:S'' \rightarrow S'$ be a $S$-morphism and suppose that the base change $i_{S''}: V_{S''} \defeq V \times_{S'} S'' \hookrightarrow X_{S''}$ factors through $U_{S''}$. Let $q \in S''$ be a point and $p = f(q)$. Notice that, in particular, the morphism between fibers $i_{S'',q}: V_{S'',q} \hookrightarrow X_{S'',q}$ factors through the fiber $U_{S'',q}$. Consider the canonical morphism $\alpha: \Spec \kappa(q) \rightarrow \Spec \kappa(p)$ and the induced commutative diagram
	\[
	\begin{tikzcd}
	& V_{S'',q} \ar[dl] \ar[dd, near start, "\alpha_V"] \ar[dr, hook, "i_{S'',q}"] & & \\[-10pt]
	U_{S'',q} \ar[dd, "\alpha_U"'] \ar[rr, hook, crossing over, near start, "j_{S'',q}"']& & X_{S'',q} \ar[dd, "\alpha_X"] \ar[r] & \Spec \kappa(q) \ar[dd, "\alpha"] \\[-20pt]
	& V_p \ar[dr, hook, "i_p"] &  &  \\[-10pt]
	U_{S',p}  \ar[rr, hook, "j_{S',p}"] & & X_{S'p}  \ar[r] & \Spec \kappa(p).
	\end{tikzcd}
	\]
	where the squares are Cartesian. Notice that since $\alpha$ is a morphism between points it is surjective, hence  the morphisms $\alpha_V,\alpha_U,\alpha_X$ are surjective. For any $v \in V_p$, let $v'' \in V_{S'',q}$ be a point such that $\alpha_V(v'')= v$. On the other hand, $i_{S'',q}(v'') = j_{S'',q}(u'')$ for some $u'' \in U_{S'',q}$. Denote $u = \alpha_U(u'')$. It follows that $j_{S',p}(u) = i_p(v)$. In other words $i_p(V_p)$ is contained in $j_{S',p}(U_{S',p})$. By \eqref{bag}, $p = f(q) \in S_0$ for all $q \in S''$. It follows that $f: S'' \rightarrow S'$ factors through $S_0 \hookrightarrow S'$.
	
	Conversely, suppose that $f: S'' \rightarrow S'$ factors through $S_0$. Let $i_{S_0}: V_{S_0} \hookrightarrow X_{S_0}$ be the base change of $i$ with respect to the inclusion $S_0 \hookrightarrow S$. Then by \eqref{bag}, we have
	\[
	i_{S_0,p}(V_{S_0,p}) \subseteq j_{S_0,p}(U_{S_0,p}) \subseteq j_{S_0}(U_{S_0})
	\]
	for all $p \in S_0$, that is, $i_{S_0}$ factors through $j_{S_0}$. It follows immediately that $i_{S''}: V_{S''} \rightarrow X_{S''}$ factors through $U_{S''} \cong U_{S_0} \times_{S_0} S''$ by the universal property of the fiber product.
\end{proof}

\begin{proposition}
	\label{bah}
	Let $S$ be a Noetherian scheme, $X$ be a separated $S$-scheme and let $i: Z \hookrightarrow X$ be a closed immersion of schemes over $S$. Then $\HHilb_{S}(Z)$ is a closed subfunctor of $\HHilb_{S}(X)$. In particular, if $X$ is projective over $S$, then $\HHilb_{S}(Z)$ is representable by a closed subscheme $\Hilb_{S}(Z) \hookrightarrow \Hilb_{S}(X)$.
\end{proposition}
\begin{proof}
	Let $\caE$ be a coherent sheaf over $X$ and for each morphism $f: S'\rightarrow S$ denote the first projection $\pr_1: X_{S'} \rightarrow X$. Recall there exists a functor
	\[
	\QQuot_{\caE/X/S}: (\cat{Noe}/S)^{op} \rightarrow \cat{Set}
	\]
	associating every morphism $f: S' \rightarrow S$ to the set of isomorphism classes of families of quotients of $\caE_{S'} \defeq  \pr_1^*\caE$, which are flat over $S'$ and have proper support, see \cite[Section 5.1.3]{nitsure-05}.
	
	In particular, when $\caE = \caO_X$, we have $\QQuot_{\caO_X/X/S} \cong \HHilb_S(X)$. If $i: Z \hookrightarrow X$ is a closed immersion, we have a canonical surjection $f:\caO_X \twoheadrightarrow \caO_Z$ over $X$. Moreover, by \cite[Lemma 5.17(ii)]{nitsure-05} we have a closed subfunctor
	\[
	\delta:\QQuot_{\caO_Z/X/S} \rightarrow \QQuot_{\caO_X/X/S}
	\]
	given on each $S$-scheme $S'$ by taking the class of the quotient $q: \pr_1^*\caO_Z \twoheadrightarrow \caF$ to the class of the quotient 
	\[
	\caO_{X_{S'}} \overset{\pr^*_1 f}{\twoheadrightarrow} \pr_1^*\caO_Z \overset{q}{\twoheadrightarrow} \caF.
	\]
	Since the support of $\caO_Z$ on $X$ is $Z$, we have a canonical isomorphism
	\[
	\QQuot_{\caO_Z/X/S} \cong \QQuot_{\caO_Z/Z/S} \cong \HHilb_S(Z)
	\]
	such that $\delta$ fits in a commutative diagram
	\[
	\begin{tikzcd}
	\QQuot_{\caO_Z/X/S} \ar[r, "\delta"] \ar[d, "\cong" description] & \QQuot_{\caO_X/X/S} \ar[d,  "\cong" description] \\
	\HHilb_S(Z) \ar[r, "\eta_i"] & \HHilb_S(X),
	\end{tikzcd}
	\]
	where $\eta_i$ is defined as in \eqref{baq} for $i: Z \hookrightarrow X$. 
\end{proof}

\begin{definition}
	\label{bai}
	Let $X$ and $Y$ be schemes over a Noetherian scheme $S$. The {\it functor of morphisms}
	\[
	\MMor_S(X,Y): (\cat{Noe}/S)^{op} \rightarrow \cat{Set}
	\]
	from $X$ to $Y$ is defined on $S$-schemes by
	\[
	\MMor_S(X,Y)(S')= \Hom_{S'}(X_{S'}, Y_{S'})
	\]
	where $\Hom_{S'}(X_{S'}, Y_{S'})$ is the set of morphisms from $X_{S'}$ to $Y_{S'}$ over $S'$. For each $S$-morphism $f: S'' \rightarrow S'$, the morphism $\MMor_S(X,Y)(f)$ is defined by base change.
	If $X$ is a separated $S$-scheme, we can define a monomorphism of functors
	\[
	\Gamma: \MMor_S(X,Y) \hookrightarrow \HHilb_S(X \times_S Y)
	\]
	by taking each $S'$-morphism $f: X_{S'} \rightarrow Y_{S'}$ to the graph $\Gamma(f_{S'}): X_{S'} \hookrightarrow (X \times_S Y)_{S'}$. If we assume $X$ is flat, then this monomorphism defines an open subfunctor, see for instance \cite[Theorem 5.23]{nitsure-05}. In particular, if $\HHilb_S(X \times_S Y)$ is representable, $\MMor_S(X,Y)$ is representable by an open subscheme of $\Hilb_S(X \times_S Y)$ denoted $\Mor_S(X, Y)$ called the \emph{scheme of $S$-morphisms} from $X$ to $Y$. This is the case if both $X$ and $Y$ are projective $S$-schemes and $X$ is flat over $S$.
\end{definition}

Let $X$, $Y$, $W$ and $S'$ be $S$-schemes and let $f: W \rightarrow Y$ be a morphism. There is a morphism of functors
\begin{equation*}
\mu_f: \MMor_S(X,W) \rightarrow \MMor_S(X,Y) 
\end{equation*}
defined on sections over $S'$ by taking any morphism $g: X_{S'} \rightarrow W_{S'}$ to the composition $f_{S'} \circ g$, where $f_{S'}$ is the base change of $f$. In fact, this defines a functor
\[
\MMor_S(X, -): \cat{Sch}/S \rightarrow \cat{Psh}(\cat{Sch}/S),
\]
where $\cat{Psh}(\cat{Sch}/S)$ is the category of presheaves of sets over $\cat{Sch}/S$.

In particular, suppose that $X,Y$ and $W$ are $S$-schemes such that both $\MMor_S(X,W)$ and $\MMor_S(X,Y)$ are representable. Then, the morphism $\mu_f$ corresponds to a morphism of schemes
\[
\Mor_S(X, f): \Mor_S(X,W) \rightarrow \Mor_S(X,Y).
\]

\begin{lemma}[Functorial properties]
	\label{bak}
	Let $S$ be a scheme. The following properties hold:
	\begin{enumerate}
		\item \label{cafn} Let $\{X_i\}_{i \in I}$ be a finite family of schemes over $S$ and $Y$ be a scheme over $S$. Then
		\[
		\MMor_S\left(\coprod_{i \in I}X_i, Y\right) \cong \prod_{i \in I}  \MMor_S(X_i, Y),
		\]
		where the right hand side denotes the fiber product over $h_S$;
		
		\item \label{cafo} Let $\cat{I}$ be a category and $\caF: \cat{I} \rightarrow \cat{Sch}/S$ be a functor such that $\lim \caF$ exists in $\cat{Sch}/S$. Then
		\[
		\MMor_S(X, \lim \caF) \cong \lim \MMor_S(X, \caF),
		\]
		where $\MMor_S(X, \caF): \cat{I} \rightarrow \cat{Psh}(\cat{Sch/S})$ is the functor $ i \mapsto \MMor_S(X, \caF(i))$. In other words, $\MMor_S(X,-)$ preserves with limits.
	\end{enumerate}
\end{lemma}
\begin{proof}
	Both properties follow straightforwardly from the analogous properties of the bifunctor $\Hom_S(-,-)$.
\end{proof}

\begin{remark}
	\label{bal}
	The following is an important particular case of Lemma \ref{bak}: let $X$ be a $S$-scheme and $Y \rightarrow Y''$, $Y' \rightarrow Y''$ be morphisms of schemes over $S$. We have a canonical isomorphism of functors
	\[
	\MMor_S(X, Y \times_{Y''} Y') \cong \MMor_S(X, Y) \times_{\MMor_S(X,Y'')} \MMor_S(X, Y').
	\]
	If we suppose that all the functors of morphisms on the right hand side are representable, then we have the corresponding isomorphism of schemes of morphisms. In particular, if $Y'' = S$ we have $\Mor_S(X,S) \cong S$, therefore
	\[
	\Mor_S(X,Y \times_S Y') \cong \Mor_S(X,Y) \times_S \Mor_S(X,Y').
	\]
\end{remark}

\begin{proposition}
	\label{ban}
	Let $X$ be a separated $S$-scheme, let $i:Z \hookrightarrow Y$ be a closed immersion and $j: U \hookrightarrow Y$ be an open immersion of schemes over $S$. Then the morphisms
	\begin{align*}
	\mu_{i}: & \MMor_S(X,Z) \rightarrow \MMor_S(X,Y), \\
	\mu_{j}: & \MMor_S(X,U) \rightarrow \MMor_S(X,Y)
	\end{align*}
	make $\MMor_S(X,Z)$ and $\MMor_S(X,U)$ respectively to be a closed and an open subfunctor of $\MMor_S(X,Y)$. In particular, if $X$ and $Y$ are projective $S$-schemes and $X$ is flat over $S$, then $\MMor(X,Z)$ and $\MMor(X,U)$ are representable and the morphisms
	\begin{equation}
	\label{cadz}
	\begin{split}
	\Mor_S(X,i):& \Mor_S(X, Z) \rightarrow \Mor_S(X,Y)\\
	\Mor_S(X,j):& \Mor_S(X, U) \rightarrow \Mor_S(X,Y)
	\end{split}
	\end{equation}
	corresponding to $\mu_i$ and $\mu_j$ are a closed and an open immersion respectively.
\end{proposition}
\begin{proof}
	First we claim that for any immersion of schemes $\iota: W \hookrightarrow X$ we claim that we have a Cartesian diagram of functors
	\begin{equation}
	\label{bam}
	\begin{tikzcd}
	\MMor_S(X,W) \ar[d, hook, "\Gamma"] \ar[r, hook, "\mu_\iota"]& \MMor_S(X,Y) \ar[d, hook, "\Gamma"] \\
	\HHilb_S(X \times_S W) \ar[r, hook, "\eta_{\iota_X}"]& \HHilb_S(X \times_S  Y)
	\end{tikzcd}
	\end{equation}
	where $\eta_{\iota_X}$ is defined as in \eqref{baq} for $\iota_X : X \times_S W \hookrightarrow X \times_S Y$.
	
	Indeed, let $S'$ be a locally Noetherian $S$-scheme and $g \in \MMor_{S}(X, W)(S')$ and denote
	\[
	g' \defeq  \mu_{\iota,S'}(g) = \iota_{S'} \circ g,
	\]
	and consider the commutative diagram
	\begin{equation}
	\label{bao}
	\begin{tikzcd}
	& X_{S'} \ar[d, hook, "\Gamma(g)"'] \ar[dr, "g"] \ar[ddd, bend right= 35, shift right = 3, "\id_{X_{S'}}"'] & \\
	& (X \times_S W)_{S'} \ar[r, "\pi"] \ar[d, hook, "\iota_{X,S'}"] & W_{S'} \ar[d, hook, "\iota_{S'}"'] \\
	& (X \times_S Y)_{S'} \ar[r, "\pr_2"] \ar[d, "\pr_1"]& Y_{S'} \ar[d] \\
	& X_{S'} \ar[r] & S',
	\end{tikzcd}
	\end{equation}
	where all squares are Cartesian. To check that \eqref{bam} commutes, it suffices to check that $\Gamma(g') =  \iota_{X, S'} \circ \Gamma(g)$, but this follows from the uniqueness of the graph morphism, since 
	\[
	\pr_1 \circ \iota_{X, S'} \circ \Gamma(g) = \id_{X_{S'}} \mathand \pr_2 \circ \iota_{X, S'} \circ \Gamma(g) = \iota_{S'} \circ g = g'.
	\]	
	Moreover, we claim that \eqref{bam} is Cartesian. It suffices to show that for each $S$-scheme $S'$ and each pair
	\[
	(V, f) \in \HHilb_S(X \times_S W)(S') \times_{\HHilb_S(X \times_S Y)(S')} \MMor_S(X,Y)(S'),
	\]
	there exists a unique morphism $g: X_{S'} \rightarrow W_{S'}$ satisfying two conditions:
	\begin{enumerate}
		\item \label{cadu} the closed subscheme defined by the closed immersion
		\[
		\eta_{\iota_X, S'}(\Gamma(g)) = \iota_{X,S'} \circ \Gamma(g): X_{S'} \rightarrow (X \times_S Y)_{S'}
		\]
		is the same as the one defined by
		\[
		V \hookrightarrow (X \times_S W)_{S'} \overset{\iota_{X,S'}}{\hookrightarrow} (X \times_S Y)_{S'};
		\]
		
		\item \label{cadv} $\mu_{\iota,S'}(g) = \iota_{S'} \circ g = f$.
	\end{enumerate}
	
	Notice that by definition of the pair $(V,f)$, we have that the closed subscheme defined by the composition
	\[
	V \hookrightarrow (X \times_S W)_S' \overset{\iota_{X,S'}}{\hookrightarrow} (X \times_S Y)_{S'}
	\]
	is the same as the one defined by $\Gamma(f): X_{S'} \hookrightarrow (X \times_S Y)_{S'}$. Therefore, we have an isomorphism $V \cong X_{S'}$. Define $g$ as the composition
	\[
	g: X_{S'} \xrightarrow{\cong} V \hookrightarrow (X \times_S W)_{S'} \xrightarrow{\pi} W_{S'}.
	\]
	By the definition of graph morphism, we obtain a commutative diagram \eqref{bao}. In other words, $g$ satisfies the conditions \ref{cadu} and \ref{cadv} above. To see that $g$ is unique, just notice that if $h$ was another morphism satisfying condition \ref{cadv}, we would have $\iota_{S'} \circ g = \iota_{S'} \circ h$, thus $g = h$, since immersions are monomorphisms in the category of schemes.
	
	To complete the proof just let $W = U$ and $\iota = j$ (resp. $W=Z$ and $\iota=i$) and by Proposition \ref{bac} (resp. \ref{bah}), we have $\mu_\iota$ is an open (resp. closed) subfunctor.
\end{proof}

\begin{proof}[Proof of Theorem \ref{aad}]
	For any quasi-projective $S$-scheme $Y$ we have an open immersion $Y \hookrightarrow Y'$ where $Y'$ is projective. By Proposition \ref{ban}, $\MMor_S(X,Y)$ is representable by an open subscheme of $\Mor_S(X,Y')$, hence the functor
	\[
	\Mor_S(X, -): \cat{QProj}/S \rightarrow \cat{Sch}/S
	\]
	is well defined. Preservation of limits, open and closed immersions follow from Lemma \ref{bak} and Proposition \ref{ban}.
\end{proof}

\section{Rational curves on blow-ups at points}
\label{daa}

Let $Y$ be a projective scheme over an algebraically closed field $k$, let $C$ be an irreducible curve over $k$ and let $\sigma: \Bl_Z(Y) \rightarrow Y$ be the blow-up of $Y$ along a closed subscheme $Z$ with exceptional divisor $E$. Consider the induced morphism
\[
\sigma_M \defeq \Mor(C, \sigma): \Mor(C, \Bl_Z(Y)) \rightarrow \Mor(C, Y),
\] and the open subschemes $\sigma^{-1}(Y \smallsetminus Z)$ and $Y \smallsetminus Z$. Recall that
\[
\sigma|_{\sigma^{-1}(Y \smallsetminus Z)}: \sigma^{-1}(Y \smallsetminus Z) \rightarrow Y \smallsetminus Z
\]
is an isomorphism. Therefore, we have a commutative diagram
\[
\begin{tikzcd}
\sigma^{-1}(Y \smallsetminus Z) \ar[r, "\sim"] \ar[d, hook] & Y \smallsetminus Z \ar[d, hook] \\
\Bl_Z(Y) \ar[r, "\sigma"] & Y.
\end{tikzcd}
\] 
If we apply the functor $\Mor(C, -)$ we obtain the following commutative diagram
\[
\begin{tikzcd}
\Mor(C, \sigma^{-1}(Y \smallsetminus Z)) \ar[r, "\sim"] \ar[d, hook] & \Mor(C , Y \smallsetminus Z) \ar[d, hook] \\
\Mor(C, \Bl_Z(Y)) \ar[r, "\sigma_M"] & \Mor(C, Y),
\end{tikzcd}
\]
where the top row is also an isomorphism and the vertical arrows are open immersions by Proposition \ref{ban}. In other words, $\sigma_M|_{\Mor(C, \sigma^{-1}(Y \smallsetminus Z))}$ is an isomorphism between open subschemes of $\Mor(C, \Bl_Z(Y))$ and $\Mor(C, Y)$.

Notice also that $\Mor(C, Z) \hookrightarrow \Mor(C, Y)$ is a closed immersion and therefore $\Mor(C, Y) \smallsetminus \Mor(C, Z)$ is an open subscheme of $\Mor(C, Y)$ parametrizing morphisms $C \rightarrow Y$ whose images intersect the open subset $Y \smallsetminus Z$. Lastly, notice that by Remark \ref{bal} we have
\[
\sigma_M^{-1}(\Mor(C, Z)) \cong \Mor(C,E).
\]

\begin{proposition}
	\label{dab}
	Let $C$ be a non-singular projective curve, $Y$ be a projective scheme over $k$, $Z$ be a closed subscheme of $Y$ and $\sigma: \Bl_Z(Y) \rightarrow Y$ be the blow-up of $Y$ along $Z$. Let
	\[
	\sigma_M: \Mor(C, \Bl_Z(Y)) \rightarrow \Mor(C, Y)
	\]
	be the induced morphism and let
	\[
	N \defeq \Mor(C, Y) \smallsetminus \Mor(C,Z)
	\]
	and
	\[
	N' \defeq  \sigma_M^{-1}(N) = \Mor(C, \Bl_Z(Y)) \smallsetminus \Mor(C, E).
	\]
	Then the restriction
	\[
	\sigma_M|_{N'}: N' \rightarrow N
	\]
	is locally quasi-finite. More specifically, it is a bijection on $k$-points. 
\end{proposition}
\begin{proof}
	Any $k$-point $[f] \in N$ corresponds to a morphism $f: C \rightarrow Y$ whose image intersects the open $Y \smallsetminus Z$. Since $\sigma|_{\sigma^{-1}(Y \smallsetminus Z)}$ is an isomorphism, there is a rational map $g: C \dashrightarrow \Bl_Z(Y)$ such that the diagram
	\[
	\begin{tikzcd}
	& \Bl_Z(Y) \ar[d, "\sigma"]\\
	C \ar[r, "f"] \ar[ru, dashed, "g"] & Y
	\end{tikzcd}
	\]
	is commutative. Since $C$ is nonsingular, $g$ is the unique regular morphism making the diagram commutative. This is equivalent to say that the fiber of the morphism $\sigma_M$ at the point $[f]$ has a unique $k$-point $[g]$.
	
	Let $\sigma^{-1}_M([f])$ denote this fiber, since it is locally of finite type, its set of $k$-points is dense and we conclude that $\sigma^{-1}_M([f])$ consists of the single point $[g]$ and hence $\sigma_{M}|_{N'}$ is a bijection on $k$-points.
	
	To see it is locally quasi-finite, recall that any irreducible component of $N'$ and of $N$ is of finite type. Let $N'_0 \subset N'$ be an irreducible component. The restriction $\sigma_M|_{N'_0}$ is bijective on $k$-points between schemes of finite type, therefore it is quasi-finite.
\end{proof}

With the notation as in Proposition \ref{dab} consider the open subscheme 
\[
U \defeq \Mor(C, \sigma^{-1}(Y \smallsetminus Z))
\]
contained in $N'$. Then the complement $N'' \defeq N' \smallsetminus U$ parametrizes morphisms from $C$ to $\Bl_Z(Y)$ whose image intersects the exceptional divisor $E$ properly.

Proposition \ref{dab} gives us a hint for the behaviour of components in $N''$. In fact, it tells us that if $N_0$ is an irreducible component of $N$, then the preimage $N'_0 \defeq \sigma^{-1}_M(N_0)$ can be roughly understood as a ``splitting'' of $N_0$. Indeed, the theorem of dimension of fibers (applied to the reduction of these schemes) implies that there exists a unique component $N'_1 \subset N'_0$ dominating $N_0$ and such that $\dim N'_1 = \dim N_0$, see for instance \cite[Proposition 5.5.1]{mustata-AG}. However, since $N'_0$ might not be irreducible, it may split in many other components of dimension strictly smaller than $\dim N_0$.

Notice that if $U \cap N'_0 \neq \emptyset$ it is isomorphic to a dense open subscheme in $N_0$, then the underlying topological space of the component $N'_1$ will be the closure of $U \cap N'_0$ in $N'_0$. Therefore, we have that any non-empty irreducible component of
\[
N''_0 \defeq N'_0 \cap (N' \smallsetminus U) \subset N''
\]
has dimension strictly less than $\dim N_0$. The behaviour of components just described is illustrated in the fibered diagram
\[
\begin{tikzcd}
\emptyset \ar[r,  "\large \cong" description] \ar[d] &[-30pt] N''_0 \cap N'_1  \ar[r, hook] \ar[d, hook] & N''_0 \ar[d, hook] \ar[r, hook] & N' \cap N'' \ar[d, hook] \ar[r, hook] & N'' \ar[d, hook] \\
U \cap N_0' \ar[r, hook] &[-30pt] N'_1  \ar[r, hook] \ar[rd, "\text{dominant}" description, near start] & N'_0 \ar[d, "\text{bijective}" description] \ar[r, hook] & N' \ar[d, "\text{bijective}" description] \ar[r, hook] & \Mor(C, \Bl_Z(Y)) \ar[d, "\sigma_M"] \\
& & N_0 \ar[r, hook] & N \ar[r, hook] & \Mor(C, Y).
\end{tikzcd}
\]
We refine this description significantly when $C = \PR^1$, $Y=\PR^n$ and $Z$ is a finite collection of points.

Consider $p = (1:0:\dotsb :0)$ and let $\sigma: \Bl_p(\PR^n) \rightarrow \PR^n$ be the blow-up of $\PR^n$ at $p$. Recall that $\Bl_p(\PR^n)$ can be defined as a closed subscheme of the product $\PR^n \times \PR^{n-1}$. Explicitly, if we define coordinates $(x_0: \dotsb : x_n)$ for $\PR^n$ and $(y_1: \dotsb: y_n)$ for $\PR^{n-1}$, the closed subscheme $\Bl_p(\PR^n)$ is given by the equations
\[
\{x_iy_j = y_jx_i \} \text{ for } 1 \le i,j \le n
\]
in $\PR^n \times \PR^{n-1}$. Therefore, we have the commutative diagram
\begin{equation}
\label{caey}
\begin{tikzcd}
\Bl_p(\PR^n) \ar[dr, hook] \ar[ddr, bend right=10, "\sigma"'] \ar[rrd, bend left =10, "\tau"]&[-10pt] & \\[-10pt]
& \PR^n \times \PR^{n-1} \ar[r] \ar[d] & \PR^{n-1} \ar[d] \\
& \PR^n \ar[r] & \Spec k.
\end{tikzcd}
\end{equation}
Recall that
\[
\Hom_k(\PR^1, \PR^n \times \PR^{n-1}) \cong \Hom_k(\PR^1, \PR^n) \times \Hom_k(\PR^1, \PR^{n-1}),
\]
that is, any morphism $\PR^1 \rightarrow \PR^n \times \PR^{n-1}$ corresponds uniquely to a pair of tuples
\begin{equation}
\label{dae}
\left((F_0:\dotsb :F_n), (G_1: \dotsb : G_n) \right)
\end{equation}
such that the collections $\{F_i\}$ and $\{G_j\}$ consist each of forms with no factors in common and the same degree (among each tuple). Thus a morphism $f: \PR^1 \rightarrow \Bl_p(\PR^n)$ corresponds uniquely to tuples \eqref{dae} that also satisfy
\begin{equation}
\label{daf}
F_iG_j = F_jG_i \text{ for } 1 \le i, j \le n.
\end{equation}

\begin{lemma}
	\label{dam}
	Let $\{p_1, \dotsc, p_r\}$ be a finite collection of distinct points in $\PR^n$. Let $\sigma: Y \rightarrow \PR^n$ be the blow-up of $\PR^n$ at $\{p_1, \dotsc, p_r\}$ with exceptional divisor $E$ and let $\Bl_{p_i}(\PR^n) \rightarrow \PR^n$ be the blow-up of $\PR^n$ with exceptional divisor $E_i$ for each $i$. Then, there exist closed immersions
	\begin{equation*}
	\iota: Y \hookrightarrow \Bl_{p_1}(\PR^n) \times_{\PR^n} \dotsb \times_{\PR^n} \Bl_{p_r}(\PR^n) \mathand E \hookrightarrow E_1 \times_{\PR^n} \dotsb \times_{\PR^n} E_r
	\end{equation*}
	such that $\sigma$ factorizes through the fibered diagram
	\[
	\begin{tikzcd}
	E \ar[r, hook] \ar[d, hook] & Y \ar[d, hook, "\iota"] \\
	E_1 \times_{\PR^n} \dotsb \times_{\PR^n} E_r \ar[d] \ar[r, hook] & \Bl_{p_1}(\PR^n) \times_{\PR^n} \dotsb \times_{\PR^n} \Bl_{p_r}(\PR^n) \ar[d] \\
	\{p_1, \dotsc, p_r\} \ar[r, hook] & \PR^n.
	\end{tikzcd}
	\]
\end{lemma}
\begin{proof}
	Let $Y_i \defeq \Bl_{p_1, \dotsc, p_i}(\PR^n)$ for $i \le r$ and let $E_{Y_i}$ be exceptional divisor of $Y_i$. It follows from \cite[Subsection B.6.9]{fulton-IT} that there is a fibered diagram
	\[
	\begin{tikzcd}
	E_{Y_{i+1}} \ar[r, hook] \ar[d, hook] & Y_{i+1} \ar[d, hook] \\ 
	E_{Y_i} \times_{\PR^n} E_{i+1} \ar[r,hook] \ar[d] & Y_i \times_{\PR^n} \Bl_{p_{i+1}}(\PR^n) \ar[d] \\
	\{p_1, \dotsc, p_{i+1}\} \ar[r,hook] & \PR^n
	\end{tikzcd}
	\]
	where the hooked arrows denote closed immersions. The induced closed immersion
	\[
	\iota: Y = Y_r \hookrightarrow Y_{r-1} \times_{\PR^n} \Bl_{p_r}(\PR^n) \hookrightarrow \dotsb \hookrightarrow \Bl_{p_1}(\PR^n) \times_{\PR^n} \dotsb \times_{\PR^n} \Bl_{p_r}(\PR^n) 
	\]
	satisfies the conditions of the statement.
\end{proof}

\begin{proof}[Proof of Theorem \ref{aac}]
	For this proof we will use the following notational convention: for any morphism of schemes $\alpha: X \rightarrow Y$ over $k$, we denote
	\[
	\alpha_M: \Mor(\PR^1, X) \rightarrow \Mor(\PR^1, Y)
	\]
	to be the corresponding morphism on schemes of morphisms.
	\vskip 0.2\baselineskip
	\noindent {\bf Case $r=1$:}
	\vskip 0.2\baselineskip
	We denote $p\defeq p_1=(1:0: \dotsb :0)$ and $Y = \Bl_p(\PR^n)$. Let $\sigma$ and $\tau$ be morphisms fitting the diagram \eqref{caey}. We can describe the scheme $\Mor(\PR^1, Y)$ by looking at the fibers of $\sigma_M$, that is, we start by noticing the partition
	\begin{equation}
	\label{dan}
	\Mor(\PR^1, \Bl_p(\PR^n)) = \coprod_{d \in \NN} \sigma_M^{-1}(\Mor_d(\PR^1, \PR^n)),
	\end{equation}
	where $\sigma^{-1}_M(\Mor_d(\PR^1, \PR^n))$ denotes the scheme theoretic inverse image of $\Mor_d(\PR^1, \PR^n)$ under $\sigma_M$. We can further partition \eqref{dan} using $\tau_M$. Indeed, for each $d \ge 0$, we have
	\begin{equation}
	\label{ccao}
	\sigma_M^{-1}\left(\Mor_d(\PR^1, \PR^n)\right) = \coprod_{e\ge 0} \sigma_M^{-1}\left(\Mor_d(\PR^1, \PR^n)\right) \cap \tau_M^{-1}(\Mor_{e}(\PR^1, \PR^{n-1})),
	\end{equation}
	where
	\[
	\sigma_M^{-1}\left(\Mor_d(\PR^1, \PR^n)\right) \cap \tau_M^{-1}(\Mor_{e}(\PR^1, \PR^{n-1}))
	\]
	denotes the scheme theoretic intersection fitting in the fibered diagram:
	\[
	\begin{tikzcd}
	\sigma_M^{-1}\left(\Mor_d(\PR^1, \PR^n)\right) \cap \tau_M^{-1}(\Mor_{e}(\PR^1, \PR^{n-1})) \ar[d, hook] \ar[r, hook] & \tau_M^{-1}(\Mor_{e}(\PR^1, \PR^{n-1})) \ar[d, hook] \\
	\sigma_M^{-1}\left(\Mor_d(\PR^1, \PR^n)\right) \ar[r, hook] & \Mor(\PR^1, \Bl_p(\PR^n)).
	\end{tikzcd}
	\]
	
	\begin{claimcounted}
		\label{dag} We have 
		\begin{align*}
		\sigma_M^{-1}(\Mor_0(\PR^1, \PR^n)) \cong \Mor_0(\PR^1, Y) \amalg \Mor_{>0}(\PR^1, E)
		\end{align*}
	\end{claimcounted}
	\begin{subproof}
		Let $f:\PR^1 \rightarrow \PR^n$ be the constant morphism sending $\PR^1$ to $p$ and let $[f]$ be the corresponding point in $\Mor_0(\PR^1, \PR^n)$. Notice that $\Mor(\PR^1, \{p\}) = [f]$ and by Remark \ref{bal}, we have a Cartesian square
		\[
		\begin{tikzcd}
		\Mor(\PR^1, E) \ar[r, hook] \ar[d] & \Mor(\PR^1, \Bl_p(\PR^n)) \ar[d, "\sigma_M"] \\
		\left[ f \right]  \ar[r, hook] & \Mor(\PR^1, \PR^n),
		\end{tikzcd}
		\]
		that is
		\[
		\sigma_M^{-1}([f]) \cong\Mor(\PR^1, E) \cong \Mor(\PR^1, \PR^{n-1}) \cong \coprod_{e \ge 0} \Mor_e(\PR^1, \PR^{n-1}).
		\]
		
		Moreover, clearly the underlying topological space of $\Mor_0(\PR^1, Y)$ is given by $\Mor_0(\PR^1, \sigma^{-1}(\PR^n\smallsetminus \{p\})) \cup \Mor_0(\PR^1, E)$.
	\end{subproof}
	
	Suppose $d > 0$. Then for any point $[g]$ in $\sigma_M^{-1}\left(\Mor_d(\PR^1, \PR^n)\right)$ the morphism $g$ corresponds uniquely to a pair of tuples \eqref{dae} satisfying \eqref{daf} and such that $\deg F_i = d$.
	
	\begin{claimcounted}
		\label{dabe}
		Let $d$ and $e$ be positive integers. The following hold:
		\begin{enumerate}
			\item \label{bar} If $e > d$, then $\sigma_M^{-1}\left(\Mor_d(\PR^1, \PR^n)\right) \cap \tau_M^{-1}(\Mor_{e}(\PR^1, \PR^{n-1})) = \emptyset;$
			
			\item \label{bas} if $e \le d$, a $k$-point $[g]$ belongs to
			\[
			\sigma_M^{-1}\left(\Mor_d(\PR^1, \PR^n)\right) \cap \tau_M^{-1}(\Mor_{e}(\PR^1, \PR^{n-1}))
			\]
			if and only if 
			\[
			\deg(\sigma \circ g) = d \mathand m_p(\sigma \circ g) = d-e;
			\]
			
			\item \label{bat} if $d = e$, then this subscheme is irreducible of dimension $(n+1)(d+1)-1$.
		\end{enumerate} 
	\end{claimcounted}
	\begin{subproof}
		Consider a morphism $g: \PR^1 \rightarrow Y$ such that $\sigma \circ g$ has degree $d$. We have two situations: either $p =(1:0: \dotsb :0)$ belongs to the image $(\sigma \circ g)(\PR^1)$ or it does not. 
		
		Let $((F_0: \dotsb: F_n), (G_1:\dotsb: G_n))$ be the pair of tuples of polynomials corresponding to $g$ as described in \eqref{dae}. Since $\sigma \circ g$ is of degree $d$ we have that all the $F_i$ are of degree $d$. Moreover, condition \eqref{daf} is equivalent to saying there exists a homogeneous polynomial $H$ such that $F_i = H \cdot G_i$ for $1 \le i \le n$.
		
		Suppose that $p \notin (\sigma \circ g)(\PR^1)$. Then $F_1, \dotsc, F_n$ in \eqref{dae} have no factors in common. Since $G_1, \dotsc, G_n$ also do not have factors in common, we have that \eqref{daf} implies
		\[
		F_i = aG_i \mathfor 1 \le i \le n \mathand a \in k.
		\]
		In particular, $\deg G_i = \deg F_i = d$, or equivalently, $\tau \circ g$ has degree $d$. In other words,
		\[
		\tau_M([g]) \in \Mor_d(\PR^1, \PR^{n-1}) \mathand [g] \in \sigma_M^{-1}\left(\Mor_d(\PR^1, \PR^n)\right) \cap \tau_M^{-1}(\Mor_{d}(\PR^1, \PR^{n-1})).
		\]
		
		Conversely, suppose that $[g] \in \sigma_M^{-1}\left(\Mor_d(\PR^1, \PR^n)\right) \cap \tau_M^{-1}(\Mor_{d}(\PR^1, \PR^{n-1}))$. Then the polynomials $F_1, \dotsc, F_n$ and $G_1, \dotsc, G_n$ all have degree $d$. Since $F_i = H \cdot G_i$ for all $1 \le i \le n$, we have that $H$ is a constant. Since $G_1, \cdots, G_n$ have no factors in common, neither do $F_1, \dotsc, F_n$, and thus, they have no roots in common. We conclude that $p \notin (\sigma \circ g)(\PR^1)$.
		
		We have proven that $[g]$ belongs to
		\begin{equation}
		\label{bau}
		\sigma_M^{-1}\left(\Mor_d(\PR^1, \PR^n)\right) \cap \tau_M^{-1}(\Mor_{d}(\PR^1, \PR^{n-1}))
		\end{equation}
		if and only if $p \not\in (\sigma \circ g)(\PR^1)$ and $\sigma \circ g$ has degree $d$. Notice that the latter condition is equivalent to $[\sigma \circ g]$ belongs to
		\begin{equation}
		\label{bav}
	 	\Mor_d(\PR^1, \PR^n \smallsetminus \{p\}) \defeq \Mor_d(\PR^1, \PR^n) \cap \Mor(\PR^1, \PR^n \smallsetminus \{p\}).
		\end{equation}
		
		It is clear that $\sigma_M$ induces a bijection between \eqref{bau} and \eqref{bav}. To see it is in fact an isomorphism, recall that the morphism
		\[
		\sigma_M|_{\Mor(\PR^1, \sigma^{-1}(\PR^n \smallsetminus \{p\}))}: \Mor(\PR^1, \sigma^{-1}(\PR^n \smallsetminus \{p\})) \rightarrow \Mor(\PR^1, \PR^n \smallsetminus \{p\})
		\]
		is an isomorphism of open subsets of $\Mor(\PR^1, Y)$ and $\Mor(\PR^1, \PR^n)$. Further, notice that $\Mor_d(\PR^1, \PR^n \smallsetminus \{p\})$ is a non-empty open subscheme of $\Mor_d(\PR^1, \PR^n)$. Indeed, since $k$ is infinite we can always find morphisms $\PR^1 \rightarrow \PR^n$ of degree $d$ whose image avoids a point $p$. Therefore, $\Mor_d(\PR^1, \PR^n \smallsetminus \{p\})$ is irreducible and nonsingular of dimension
		\[
		\dim \Mor_d(\PR^1, \PR^n) = (n+1)(d+1)-1,
		\]
		which proves item \eqref{bat} of the claim.
		
		Now suppose
		\[
		p = (1: 0 : \dotsb : 0) \in (\sigma \circ g)(\PR^1).
		\]
		Since $d > 0$, we have that $p \in (\sigma \circ g)(\PR^1)$ if and only if the polynomials $\{F_i\}_{1 \le i \le n}$ have roots in common or equivalently they have common factors. Let $H$ be a form in $k[u,v]$ such that
		\[
		F_i = H \cdot F_i' \mathfor 1 \le i \le n
		\]
		(notice that $H \nmid F_0$) and such that $m\defeq \deg H$ is maximal. In such a situation we have equalities $F_iG_j = F_jG_i$ for all $1 \le i,j \le n$ if and only if
		\[
		G_i = aF'_i \mathfor 1 \le i \le n \mathand a \in k.
		\]
		In particular, all the polynomials $\{G_i\}_{1 \le i \le n}$ have degree $e\defeq d-m$, that is, $\tau \circ g$ has degree $e$ and, by definition, $m_p(\sigma \circ g) = m = d - e$. In other words,
		\begin{align*}
		\tau_M([g]) & \in \Mor_{d-m}(\PR^1, \PR^{n-1}) \mathand  \\
		[g] & \in \sigma_M^{-1}\left(\Mor_d(\PR^1, \PR^n)\right) \cap \tau_M^{-1}(\Mor_{d-m}(\PR^1, \PR^{n-1}))
		\end{align*}
		for some $1 \le m < d$, in particular $e < d$. This and the case $p \notin (\sigma \circ g)(\PR^1)$ prove item \eqref{bas} of the claim. And since we proved $e \le d$ we obtain \eqref{bar} of the claim.
	\end{subproof}
	
	Define
	\[
	M_{d, m} \defeq  \sigma_M^{-1}\left(\Mor_d(\PR^1, \PR^n)\right) \cap \tau_M^{-1}(\Mor_{d-m}(\PR^1, \PR^{n-1}))
	\]
	for all $d \ge 1$ and $0 \le m < d$. Notice that by Claim \ref{dabe}, $\deg (\tau \circ g) = d-m$ if and only if we have parametric multiplicity $m_{p}(\sigma \circ g) = m$. Thus, $M_{d,m}$ has $k$-points satisfying the conditions of the statement of the theorem.

	The partitions \eqref{ccao} and the one given by Claim \ref{dag} conclude the proof of items \eqref{dah}, \eqref{dai}, \eqref{daj} of the theorem for $r=1$.
	
	\vskip 0.2\baselineskip
	\noindent {\bf Case $r>1$:}
	\vskip 0.2\baselineskip
	We fix the following notation:
	\begin{itemize}
		\item $\sigma: Y = \Bl_{p_1, \dotsc, p_r}(\PR^n) \rightarrow \PR^n$ is the blow-up of $\PR^n$ along $\{p_1, \dotsc, p_r\}$ with exceptional divisor $E$;
		\item $\sigma_i: \Bl_{p_i}(\PR^n) \rightarrow \PR^n$ is the blow-up of $\PR^n$ at $p_i$ with exceptional divisor $E_i$;
		\item $\pr_i: \Bl_{p_1}(\PR^n) \times_{\PR^n} \dotsb  \times_{\PR^n} \Bl_{p_r}(\PR^n) \rightarrow \Bl_{p_i}(\PR^n)$ is the natural $i$-th projection;
		\item $\theta: \Bl_{p_1}(\PR^n) \times_{\PR^n} \dotsb  \times_{\PR^n} \Bl_{p_r}(\PR^n) \rightarrow \PR^n$ is the natural morphism to $\PR^n$;
		\item $\iota: Y \hookrightarrow \Bl_{p_1}(\PR^n) \times_{\PR^n} \dotsb \times_{\PR^n} \Bl_{p_r}(\PR^n)$ is the closed immersion defined in Lemma \ref{dam}.
		\item $N \defeq \Mor(\PR^1, \PR^n)$ and $N_{>0} \defeq \Mor_{>0}(\PR^1, \PR^n)$;
		\item The partition
		\[
		\Mor_{>0}(\PR^1, \Bl_{p_i}(\PR^n)) \cong \left( \coprod_{d_i \in \NN} \coprod_{m_i<d_i} M_{d_i,m_i}\right) \amalg \Mor_{>0}(\PR^1, E_i)
		\] induced by each $\sigma_i$ by the case $r = 1$.
		\item $f_i: \PR^1 \rightarrow \PR^n$ is the constant morphism such that $f_i(\PR^1) = p_i$.
	\end{itemize}
		
		Recall that we have a commutative diagram
		\[
		\begin{tikzcd}
		\Mor(\PR^1, Y) \ar[r, hook, "\iota_M"] \ar[dr, "\sigma_M"']& \Mor(\PR^1, \Bl_{p_1}(\PR^n) \times_{\PR^n} \dotsb  \times_{\PR^n} \Bl_{p_r}(\PR^n)) \ar[d, "\theta_M"]\\
		&
		N
		\end{tikzcd}
		\]
		and moreover, we have that $\Mor_{>0}(\PR^1, Y) \smallsetminus \Mor_{>0}(\PR^1, E)$ is mapped to $N_{>0}$ via $\sigma_M$ and
		\[
		\sigma_{i,M}^{-1}(N_{>0}) = \coprod_{d_i >0} \coprod_{m_i<d_i} M_{d_i,m_i}.
		\]
		Since coproducts commute with fibered products in the category of schemes and $\sigma_i \circ \pr_i = \theta$ for all $i$, we have
		\begin{equation*}
		\theta^{-1}_M(N_{>0}) =  \coprod_{d_1, \dotsc, d_r > 0}\coprod_{m_i < d_i } \pr_{1,M}^{-1}(M_{d_1,m_1}) \cap \dotsc \cap \pr_{r,M}^{-1}(M_{d_r,m_r}).
		\end{equation*}
		
		Now let $[g] \in \Mor(\PR^1, Y)$ be a $k$-point and suppose that
		\[
		\pr_{i,M} \circ \iota_M([g]) = [\pr_i \circ \iota \circ  g] \in M_{d_i, m_i}.
		\]
		Then, we have that
		\[
		\sigma_{i,M} \circ \pr_{i,M} \circ \iota_M([g]) = \theta_M \circ \iota_M([g])
		\]
		corresponds to a morphism
		\[
		\theta \circ \iota \circ g: \PR^1 \rightarrow \PR^n
		\]
		of degree $d_i$. We conclude that
		\[
		\iota_M^{-1}\left(\pr_{i,M}^{-1}\left(M_{d_i,m_i}\right) \cap \pr_{j,M}^{-1}\left(M_{d_j,m_j}\right)\right) = \emptyset
		\]
		whenever $d_i \neq d_j$. Thus, we obtain the partition
		\begin{equation*}
		\begin{split}
		\sigma_{M}^{-1}(N_{>0}) = \iota_M^{-1} \left(\theta_M^{-1}(N_{>0})\right) & = \\
		\coprod_{d \in \NN}\coprod_{m_i < d} & \iota_M^{-1}\left(\pr_{1,M}^{-1}\left(M_{d,m_1}\right)\right) \cap \dotsc \cap \iota_M^{-1}\left(\pr_{r,M}^{-1}\left(M_{d,m_r}\right)\right)
		\end{split}
		\end{equation*}
	
	For each positive integer $d$ and each collection $\mathbf{m} = (m_1, \dotsc, m_r)$ such that $m_i<d$, define
	\[
	M_{d, \mathbf{m}}\defeq  \iota_M^{-1}\left(\pr_{1,M}^{-1}\left(M_{d,m_1}\right)\right) \cap \dotsc \cap \iota_M^{-1}\left(\pr_{r,M}^{-1}\left(M_{d,m_r}\right)\right).
	\]
	Hence, for any point $k$-point $[g] \in M_{d, \mathbf{m}}$ and $1 \le i \le r$ we have
	\[
	\pr_{i,M}\circ \iota_M ([g]) = [\pr_i \circ \iota \circ g] \in M_{d,m_i}.
	\]
	On the other hand, by definition of $M_{d,m_i}$ in the case $r=1$, this happens if and only if
	\begin{equation*}
	\begin{split}
	\deg(\sigma_i \circ \pr_i \circ \iota \circ g) & = \deg(\sigma \circ g) = d \mathand \\
	m_{p_i}(\sigma_i \circ \pr_i \circ \iota \circ g) & = m_{p_i}(\sigma \circ g) = m_i.
	\end{split}
	\end{equation*}
	
	Recall that for each $i$ the component $M_{d,0}$ contained in $\Mor(\PR^1, \Bl_{p_i}\PR^n)$ is equal to $\sigma_{i,M}^{-1}\left(\Mor_d(\PR^1, \PR^n \smallsetminus \{p_i\})\right)$. If we denote $\mathbf{0} = (0,\dotsc, 0)$, then, since fibered products commute with fibered products and $\Mor(\PR^1, -)$ commutes with fibered products (and, in particular, intersections), it follows that $M_{d,\mathbf{0}}$ is equal to
	\begin{align*}
	\iota_M^{-1} & \left(\pr_{1,M}^{-1} \left(\sigma_{1,M}^{-1}\left(\Mor_d(\PR^1, \PR^n \smallsetminus \{p_1\})\right)\right)\right. \cap \dotsb  \\ 
	& \qquad \left. \cap \pr_{r,M}^{-1}\left(\sigma_{r,M}^{-1}\left(\Mor_d(\PR^1, \PR^n \smallsetminus \{p_r\})\right)\right)\right) \\
	= & \iota_M^{-1}\left( \theta_M^{-1}\left(\Mor_d(\PR^1, \PR^n \smallsetminus \{p_1\})\right) \cap \dotsb \cap \theta_M^{-1}\left(\Mor_d(\PR^1, \PR^n \smallsetminus \{p_r\})\right)\right) \\
	= & \sigma_M^{-1}\left(\Mor_d(\PR^1, \PR^n \smallsetminus \{p_1, \dotsc p_r\} )\right).
	\end{align*}
	
	Since $\sigma$ is an isomorphism when restricted to the preimage of $\PR^n \smallsetminus \{p_1, \dotsc, p_r\}$, we have $M_{d, \mathbf{0}}$ is isomorphic to an open subset of $\Mor_d(\PR^1, \PR^n)$ therefore it is nonsingular and irreducible of dimension $(n+1)(d+1)-1$. Here we use that $k$ is infinite to guarantee $\Mor_d(\PR^1, \PR^n \smallsetminus \{p_1, \dotsc, p_r\})$ is always non-empty.
	
	To finish the description of the partition of the statement of the theorem it suffices to describe non-constant morphisms $g:\PR^1 \rightarrow Y$ such that $\sigma \circ g$ is constant, this happens if and only if $\sigma \circ g = f_i$ for some $i$. We have the fibered diagram
	\[
	\begin{tikzcd}
	E \ar[r, hook] \ar[d, hook] & Y \ar[d, hook, "\iota"] \\
	\theta^{-1}(\{p_1, \dotsc, p_r\}) \ar[r, hook] \ar[d] & \Bl_{p_1}(\PR^n) \times_{\PR^n} \dotsb \times_{\PR^n} \Bl_{p_r}(\PR^n) \ar[d, "\theta"] \\
	\{p_1, \dotsc, p_r\} \ar[r, hook] & \PR^n
	\end{tikzcd}
	\]
	inducing the corresponding fibered diagram
	\[
	\begin{tikzcd}
	\Mor(\PR^1, E) \ar[r, hook] \ar[d, hook] & \Mor(\PR^1, \Bl_{p_1, \dotsc, p_r}(\PR^n)) \ar[d, hook, "\iota_M"] \\
	\theta^{-1}_M(\{[f_1], \dotsc, [f_r] \}) \ar[r, hook] \ar[d] &  \Mor(\PR^1, \Bl_{p_1}(\PR^n)) \times_{N} \dotsb  \times_{N} \Mor(\PR^1,\Bl_{p_r}(\PR^n)) \ar[d, "\theta_M"]\\
	\{[f_1], \dotsc, [f_r] \} \ar[r, hook] &  N.
	\end{tikzcd}
	\]
	Since $E_i \cong \sigma^{-1}(p_i)$, we have $E = \coprod_{i=1}^r \sigma^{-1}(p_i) \cong \coprod_{i=1}^r E_i$. Since fibered products commute with coproducts, it follows that
	\begin{align*}
	\Mor(\PR^1, E) & \cong \iota_M^{-1}\left(\coprod_{i=1}^r \theta^{-1}_M([f_i]) \right) \cong  \coprod_{i=1}^r \sigma^{-1}_M([f_i]) \cong \coprod_{i=1}^r \Mor(\PR^1, E_i) \\
	& \cong \coprod_{i=1}^r \coprod_{e \ge 0} \Mor_e(\PR^1, E_i)
	\end{align*}
	
	Therefore we obtain the partition of the statement, which finishes the proof of \eqref{dah},\eqref{dai} and \eqref{daj} of the theorem.
	
	To prove \eqref{dak} notice that for any positive $d$ we have
	\[
	\sigma^{-1}_M(\Mor_d(\PR^1,\PR^n)) \cong \coprod_{m_i < d} M_{d,\mathbf{m}}.
	\]
	By Proposition \ref{dab} we have that $\sigma_M|_{\sigma^{-1}_M(\Mor_d(\PR^1,\PR^n))}$ is a bijection on $k$-points, in particular, all of its fibers are irreducible of dimension $0$. Hence, by the theorem on dimension of fibers, there exists a unique component of $\sigma^{-1}_M(\Mor_d(\PR^1,\PR^n))$ dominating $M_{d,\mathbf{m}}$ and of dimension $(n+1)(d+1)-1$. We already know this component is $M_{d,\mathbf{0}}$, hence if $\mathbf{m} \neq \mathbf{0}$ every irreducible component of $M_{d,\mathbf{m}}$ has dimension strictly smaller than the dimension of $M_{d,\mathbf{0}}$.
\end{proof}

\section{Conflicts of interest}

The author states that there is no conflict of interest.

\printbibliography

\end{document}